\documentclass[11pt]{article} 
\usepackage[utf8]{inputenc} 
\usepackage{amsmath}
\usepackage{amsthm}
\usepackage{amsfonts} 
\usepackage{geometry} 
\usepackage{setspace}
\usepackage{cite}
\newtheorem{thm}{Theorem}
\newtheorem{lem}{Lemma}[section]
\newtheorem{prop}{Proposition}[section]
\newtheorem* {remark}{Remark}

\usepackage{color}
\usepackage{booktabs} 
\usepackage{array} 
\usepackage{paralist} 
\usepackage{verbatim} 
\usepackage{subfig}
\usepackage{sectsty}
\usepackage{bm}
\allsectionsfont{\sffamily\mdseries\upshape} 
\usepackage[nottoc,notlof,notlot]{tocbibind} 
\usepackage[titles,subfigure]{tocloft}

\newcommand{\beq}{\begin{equation}}
\newcommand{\eeq}{\end{equation}}
\newcommand{\mc}{\mathcal}
\newcommand{\mb}{\mathbb}
\newcommand{\fc}{\frac}
\newcommand{\nid}{\noindent}
\newcommand{\ra}{\rightarrow}
\renewcommand{\l}{\left}
\renewcommand{\r}{\right}
\newcommand{\Z}{\mathbb{Z}}

\newcommand{\R}{\mathbb{R}}
\newcommand{\T}{\mathbb{T}}
\newcommand{\D}{\mathcal{D}}
\newcommand{\C}{\mathcal{C}}
\renewcommand{\a}{\alpha}

\renewcommand{\b}{\beta}

\newcommand{\e}{\epsilon}

\newcommand{\ba}{\begin{array}}
\newcommand{\ea}{\end{array}}
\renewcommand{\d}{\delta}

\newcommand{\bal}{\begin{aligned}}
\newcommand{\eal}{\end{aligned}}

\newcommand{\Nki}{N^{\fc{1}{4}}|k_i|^{\fc{3}{4}}}
\newcommand{\Nkj}{N^{\fc{1}{4}}|k_j|^{\fc{3}{4}}}
\newcommand{\ehalf}{\epsilon^{\fc{1}{2}}}

\newcommand{\Qpmi}{Q_{p^{(i)},\mathbf{m}^{(i)}}}

\makeatletter 
\@addtoreset{equation}{section}
\makeatother

\title{Ergodic Deviations of Degenerate Multidimensional Actions - Convex Bodies}
\author{Hao Wu}
\date{}

\begin{document}

\maketitle
\textbf{Abstract.} We prove that the ergodic deviation of a degenerate $\mathbb{Z}^2$-action on the torus $\mathbb{T}^2$ relative to a symmetric, strictly convex body can be decomposed into two parts, and that each part admits a limit distribution after choosing a suitable normalizer. Specifically, the first part is similar to an ergodic sum of smooth observables after being normalized by $N$, and the second part is similar to the case of a random toral translation, i.e., the $\mathbb{Z}$-action, but with a normalizer of $N^{\frac{1}{2}}$. The key difference is that we employ the product flow on the product space of $\Z^2$ lattices for the multidimensional action.

\clearpage
\tableofcontents
\clearpage
\begin{spacing}{1.2}
\section{INTRODUCTION}
In a $d$-dimensional torus, given a translation vector $\a=(\a_1,\dots, \a_d)\in \R^d$, we can consider the dynamical system $(\T^{d},T_{\a},\mu) $, where $\mu$ is the Haar measure on $\T^d$, and $T_{\a}$ is the translation from $\T^d\rightarrow \T^d$ defined by $T(x)=x+\a$, in the sense of modulo 1 for each coordinate. In this dynamical system, ergodic theory states that for every irrational translation, the number of visits inside a measurable set $\C$ before time $N$ has a ratio conveging to the measure of the set $\text{Vol}(\C)$. One object of interest is the discrepancy function defined as the difference of the actual visits before time $N$ and the expected visits $N\text{Vol}(\C)$. In dimension 1, Kesten \cite{Kesten1,Kesten2} proved that the discrepancy function for the circle rotation relative to an interval converges to a standard Cauchy distribution after being normalized by $\log N$.

There are different ways to extend this result to higher dimensions, one way is to study the random toral translation relative to higher dimensional counterparts of the interval, such as balls (analytic convex bodies) and boxes, both of which were studied by Dolgopyat and Fayad in \cite{dolgopyat2014deviations,dolgopyat2012deviations}. They showed that $d$-dimensional boxes behave similarly to 1-dimensional intervals, i.e., the discrepancy function also converges to a Cauchy distribution after normalized by $(\ln N)^d$. As for balls, they showed that the discrepancy function converges to a distribution function defined over the product space of infinite tori and the homogeneous space $\text{SL}(d+1,\R)/\text{SL}(d+1,\Z)$ with the normalizer $N^{(d-1)/2d}$. Their proof consists of a combination of harmonic analysis of the discrepancy's Fourier series, probability, and an important ingredient is the equidistribution of locally unstable submanifolds over the whole space of unimodular lattices.

In this paper, we follow a similar approach as Dolgopyat and Fayad, but instead of translations, we will consider the degenerate $\Z^2$ action in dimension $d=2$ (see the definition below), we restrict the set to be strictly convex, symmetric, and analytic bodies $\C$. Given a convex body $\C$, we denote $\C_r$ the rescaled bodies with ratio $r>0$ by the homothety centered at the origin, where $r<r_0$ so that $\C_r$ can fit inside the unit cube of $\R^2$. Let $\a=(\a_1,\a_2)\in\T^2$ be the action vector, we consider the following discrepancy function:

\beq\label{discrepancy}
D_\mathcal{C}(r,x, \alpha;N)=\sum_{\substack{0\le n_1\le N-1\\ 0\le n_2\le N-1}} \chi_{\mathcal{C}_r}(x_1+n_1\alpha_1, x_2+n_2\alpha_2)-N^2\text{Vol}(\mathcal{C}_r)
\eeq
where $\chi_{\C_r}$ is the indicator function of the set $\C_r$. 

We will show that by decomposing the discrepancy function into $2$ components, each component would admit a limit distribution after a suitable normalization, specifically:
\begin{equation}\label{decomposition}
D_{\C}(r,x,\alpha;N)=\sum_{\bar{d}=1}^2 D_{\C,\bar{d}}(r,x,\alpha;N)
\end{equation}
where $D_{\C,\bar{d}}$ represents the part of the Fourier series of $D_{\C}$ with coefficients of $\bar{d}$ non-zero coordinate(s), whose definitions will be clearer after we introduce the Fourier series of $\D_{\C}$ in Section 3.

Our main result is the following:

\begin{thm}\label{the main thm about limit law}
Let $\mathcal{C}$ be a symmetric, strictly convex analytic body that fits inside the unit cube of $\mathbb{R}^2$, and $D_{\C}$, $D_{\C, \bar{d}}$ defined as in \eqref{discrepancy} and \eqref{decomposition}, there exists a limit distribution for each $D_{\C,\bar{d}}(r,x,\a)$ after a suitable normalization, specifically, we have 2 distinct cases:

(a) For $\bar{d}=1$, assume that $(x,\a)$ are uniformly distributed in $\T^2\times \T^2$, then for every fixed $r$, there exists a function $\mc{D}_{\C, 1,r}(x,\a,\b) : (\T^2)^3 \rightarrow \R $, such that as $N\rightarrow \infty$, 
$$
D_{\C,1} (r,x,\alpha;N)/N \Rightarrow \mc{D}_{\C, 1,r}(x,\a,\b)
$$
in distribution, where $(x,\a,\b)$ is uniformly distribtuted on $(\T^2)^3$.

(b) For $\bar{d}=2$, assume that $(r,x,\a)$ are uniformly distributed in $X=[a,b] \times \T^2\times \T^2$, and denote $\lambda$ the normalized Lebesgue measure on $X$, then there exists a distribtuion function $\mathcal{D}_{\mathcal{C},2}(z):\mathbb{R}\rightarrow [0,1]$ such that for any $b>a>0$, we have 
\begin{equation}\label{limit in the thm 1}
\lim _{N \rightarrow \infty} \lambda\{(r,x,\alpha)\in [a,b]\times \mathbb{T}^2 \times \mathbb{T}^2|\frac{D_{\mathcal{C},2}(r,x,\alpha;N)}{r^{\frac{1}{2}} N^{\fc{1}{2}}}\le z\} = \mathcal{D}_{\mathcal{C},2}(z).
\end{equation}
\end{thm}

The explicit forms of $\mathcal{D}_{\mathcal{C},\bar{d}}$ will be given in Proposition \ref{Prop related to thm} of Section 2.

\begin{remark}
With the same reason as in \cite{dolgopyat2014deviations,dolgopyat2012deviations}, $r$ needs to be random in part (b) in Theorem \ref{the main thm about limit law}, in order to help prove independence between variables. While in part (a), the function behaves like a smooth function, and $r$ does not need to be random.
\end{remark}

This paper is organized as follows: Section 2 will present the explicit form of the distribution functions. In Section 3 we will prove the limit distribution of the easier part of the discrepancy function $D_{\C,1}$. Sections 4 to 6 are devoted to the general d-dimensional counterpart of the sum $D_{\C,2}$, we give a detailed description of the
sum in terms of short vectors of the lattice spaces, and how the variables become independent as $N\rightarrow \infty$. Section 4 obtains the main part of the sum that contributes to the discrepancy by using harmonic analysis. Section 5 introduces the space of lattices and express the discrepancy in the language of lattices. Section 6 shows the variables in the expression of Section 5 become independent as $N\rightarrow \infty$.

\section{LIMIT DISTRIBUTIONS}
\subsection{Limit Dstribution for the case $\bar{d}=1$.}
\begin{prop}\label{Limit distribution for d small}
If $\mc{C}$ is an analytic symmetric strictly convex body in $\mb{R}^2$, then we have 
$$
\mc{D}_{\mc{C},1}(r,x,\a,\b)=B_{\C_r}(\a,\b)-B_{\C_r}(\a,x),
$$
where
$$
B_{\C_r}(\a,x)=\sum_{k\neq 0} \fc{a_k(r)}{e^{2\pi i(k,x)}-1}e^{2\pi i(k,x)},
$$
$a_k(r)=0$ when $k_1k_2\neq 0$ and $a_k(r)=\hat{\chi}_{\mc{C}_r}(k)$ when $k_1k_2= 0$, where
$\hat{\chi}_{\mc{C}_r}(k)$ represent the $k$th Fourier coefficient of $\chi_{\mc{C}_r}$, the specific form of which is shown in \eqref{Fourier series of C}. 
\end{prop}

\subsection{Limit Distribution for the case $\bar{d}=2$.}
\textbf{Notations.} Let $M=SL(2,\mb{R})/SL(2,\mb{Z})$ denote the space of 2-dimensional unimodular lattices of $\mb{R}^2$. $M^2=\prod_\text{2 copies} M$. Given $L=(L_1, L_2) \in M^2$ we denote by $e_1(L_i)$ the shortest vector in $L_i$, and $e_2(L_i)$ the shortest vector in $L_i$ among those who have the shortest nonzero projection on the orthocomplement of the line generated by $e_1(L_i)$. Clearly the vectors $e_1(L_i), e_2(L_i)$ are well defined outside a set of Haar measure $0$. In fact, these vectors generate the lattice (see \cite{arnol2013mathematical}). We denote $e(L_i)=(e_1(L_i),e_2(L_i))$.

Let $\mc{Z}$ be the set of prime vectors $m\in \mb{Z}^2$ (i.e. the coordinates are coprime) with positive first nonzero coordinate. For later usage in Section 4 and 5, we define $\mc{Z}^2=\{\mathbf{m}=(m^1,m^2), m^i\in \mc{Z}\}$, and let 
$$
T_2^{\infty}=(\mb{T}^{2})^2 \times\mb{T}^{\mathcal{Z}\times\mc{Z}^2}
$$
We denote elements of $T_2^\infty$ by $(\bm{\theta}, \mathbf{b})$, where $\bm{\theta}=(\theta^1, \theta^2)$, $\theta^i\in \mb{T}^2$, and $\mathbf{b}=(b_{p, \mathbf{m}})_{(p, \mathbf{m})\in {\mathcal{Z}\times\mc{Z}^2}}$. For $\mathbf{m}=(m^1,m^2) \in \mc{Z}^2$ and $L=(L_1,L_2)\in M^2$, we denote by $(X_{m^i},Z_{m^i})=(m^i, e(L_i))$ the vector $m^i_1e_1(L_i)+m^i_2e_2(L_i)$. Given a prime vector $p=(p_1,p_2)\in \mathcal{Z}$, we denote $X_{p,\mathbf{m}}=(p_1X_{m^1},p_2 X_{m^2})$ and $R_{p,\mathbf{m}}=\|X_{p,\mathbf{m}}\|$ the Euclidean norm of $X_{p,\mathbf{m}}$. 
\\

\nid\textbf{Limit distribution.} Let $\mc{C}$ be a stricly convex body with smooth boundary. For each vector $\xi \in \mb{S}^{1}$, we denote by $K(\xi)$ the gaussian curvature of $\partial C$ at the unique point $x(\xi)\in \partial C$ where the unit outer normal vector is $\xi$.

Denote
$$
\mc{M}_2=M^2\times T_2^{\infty},
$$
and let $\mu$ be the Haar measure on $\mc{M}_2$. Define the following function on $\mc{M}_2$
\beq\label{limit distribution law}
\begin{aligned}
&\mc{L}_{\mc{C}} \l(L,\bm{\theta}, \mathbf{b}\r)= \fc{1}{\pi^{3}}\sum_{|\check{p}|=1}^{\infty} \sum_{p\in \mathcal{Z}} \sum_{m\in \mc{Z}^2} K^{-\fc{1}{2}}\l(\fc{X_{p,\mathbf{m}}}{R_{p,\mathbf{m}}}\r) \\
&\times\fc{\cos\l(2\pi\check{p}\l(\sum_{i=1}^2 \l(p_i\l(m^i,\theta^i\r)\r)\r)\r)\sin\l(2\pi\l(\check{p}b_{p,\mathbf{m}}-\fc{1}{8}\r)\r)\prod_{i=1}^{2}\sin\l(\pi\check{p}p_i Z_{m^i}\r)}{\check{p}^{\fc{7}{2}}R_{p,\mathbf{m}}^{\fc{3}{2}}\prod_{i=1}^{2}\l(p_iZ_{m^i}\r)}
\end{aligned}.
\eeq
The distribution $\mc{D}_{\mc{C}, 2}$ of Theorem 1 can now be described by the level set of the function above:

\begin{prop}\label{Prop related to thm}
If $\mc{C}$ is an analytic, symmetric, strictly convex body in $\mb{R}^2$, then for any $z\in \mb{R}$ we have 
$$
\mc{D}_{\mc{C},2}(z)=\mu\{(L,\bm{\theta}, \mathbf{b})\in \mc{M}_2:\mc{L}(L,\bm{\theta}, \mathbf{b})\le z\}.
$$
\end{prop}

\section{FOURIER SERIES AND PROOF OF PROPOSITION \ref{Limit distribution for d small}}

In this section, we first introduce the Fourier series of the discrepancy function as our main object of study, and then give a proof for the limit distribution of the first part of the discrepancy function.

In following sections, $\epsilon>0$ is fixed and can be arbitrarily small. The constants $C$ may vary between inequalities but it does not depend on any variables other than the dimension $d$, which is fixed to 2 in our case. Though we only treat the special case $d=2$ for the sum $D_{\C}$, it will be clear from the proof that higher dimensional counterparts $D_{\C,\bar{d}<\fc{d+1}{2}}$ can be treated in the exact same way as $D_{\C,1}$, and higher dimensional counterparts $D_{\C,\bar{d}>\fc{d+1}{2}}$ is similar to the case of $D_{\C,2}$. The sum $D_{\C,\fc{d+1}{2}}$ exhibits distinctly different behavior and we conjecture that it is similar to the case of toral translations relative to boxes and admits a limit Cauchy distribution. 
\subsection{Fourier series for convex bodies.}
We introduce the Fourier series for the smooth strictly convex body $\mathcal{C}$ by using the aymptotic formula obtained in \cite{herz1962fourier}. For each vector $t\in\mathbb{R}^2$, define its maximal projection on $\partial \mathcal{C}$ by $P(t)=\sup_{x\in\partial \mathcal{C}}(t,x)$, if $\mathcal{C}$ is of class $C^{\fc{9}{2}}$, then we have the following formula for the Fourier node:

\begin{equation}\label{Fourier series of C}
(2\pi i |t|) \hat{\chi}_{\mathcal{C}}(t)=\rho(\mathcal{C},t)-\bar{\rho}(\mathcal{C},-t)
\end{equation}

with
\[
\rho(\mc{C},t)=|t|^{-\frac{1}{2}}K^{-\frac{1}{2}} (t/|t|) e^{i2\pi(P(t)-\fc{1}{8})}+\mathcal{O}(|t|^{-\fc{3}{2}}).
\]

By a change of variable we have $\hat{\chi}_{\mc{C}_r}(k)=r\hat{\chi}_{\mc{C}}(rk)$, and by grouping the corresponding positive and negative terms in the Fourier series we get that for a symmetric body:
\beq \label{Fourier term for symmetric body}
\chi_{\mc{C}_r}(x)-Vol(\mc{C}_r)=r^{\frac{1}{2}}\sum_{k\in{\mb{Z}^2}-\{0\}} c_k(r)\cos(2\pi(k,x)),
\eeq
$$
c_k(r)=d_k(r)+\mc{O}(|k|^{-\fc{5}{2}}),
$$
$$
d_k(r)=\frac{1}{\pi}\frac{g(k,r)}{|k|^{\fc{3}{2}}},
$$
$$
g(k,r)=K^{-\frac{1}{2}}(k/|k|)\sin(2\pi(rP(k))-\fc{1}{8}).
$$

\subsection{Proof for the limit distribution when $\bar{d}=1$.}
We will show that after being normalized by $N$, $D_{\C,1}$, the part of the Fourier series that consists of nodes
$$k=(k_1,k_2)\neq (0,0) \text{ and } k_1k_2=0,$$ 
will behave like the ergodic sum of a smooth function.

First, we define
$$
A_{\C_r}(x)=\sum_{k_1\neq 0}\hat{\chi}_{\C_r}(k_1,0)e^{i2\pi k_1x_1} + \sum_{k_2\neq 0}\hat{\chi}_{\C_r}(0,k_2) e^{i2\pi k_2x_2}=: \sum_{k\in{\Z^2}-{0}}a_k(r)e^{i2\pi(k,x)},
$$
where
$a_k(r)=0$ when $k_1k_2\neq 0$ and $a_k(r)=\hat{\chi}_{\mc{C}_r}(k)$ when $k_1k_2= 0$,
then $D_{\C,1}$ takes the following form:
$$
D_{\C,1}(r,x,\alpha;N)=N\sum^{N-1}_{n=0} A_{\C_r}(x+n\a),
$$
Proposition \ref{Limit distribution for d small} will follow if we could prove the following:
\begin{lem}
For almost every $\a\in \mb{T}^{2}$, the series defined by:
$$
B_{\C_r}(\a,x)=\sum_{k\neq 0} \fc{a_k(r)}{e^{2\pi i(k,\a)}-1}e^{2\pi i(k,x)}, 
$$
is convergent in $L^2(x)$, and we have 
$$
A_{\C_r}(x+n\a)=B_{\C_r}(\a,x+(n+1)\a)-B_{\C_r}(\a,x+n\a).
$$
\end{lem}
\begin{proof}
The identity is obtained by direct calculation. We will focus on the convergence of the series $B_{\C_r}(\a,x)$. Note that 
$$
\int_{\T^2} |B_{\C_r}(\a,x)|^2 dx\le C \l(\sum_{k_1\neq 0} \fc{1}{|k_1|^3|e^{i2\pi k_1\a_1}-1|^2}+\sum_{k_2\neq 0} \fc{1}{|k_2|^3|e^{i2\pi k_2\a_2}-1|^2}\r)
$$
Therefore it suffices to prove that the series
\beq\label{series of convergence in L2}
\sum_{k_i\neq 0} \fc{1}{|k_i|^3\|k_i\a_i\|^2}
\eeq
is convergent for almost every $\a_i\in \T$, $i=1,2$.

By standard application of Borel-Cantelli Lemma, we have for almost every $\a_i\in \T$, every $k_i>0$ and every $\d>0$ we have
\beq\label{inequality for (k_ia_i)}
\|k_i\a_i\|\ge \fc{C(\a_i,\d)}{|k_i|(\ln |k_i|)^{1+\d}},
\eeq
which gives 
\beq\label{inequality for ln(k_ia_i)}
|\ln\|k_i\a_i\||\le C\ln|k_i|
\eeq
where for convenience, $\ln 1$ is defined as 1.
Therefore by taking $\d$ small, and let the constant $C(\a,\d)$ vary from line to line,
\beq\label{inequality for d=1}
\bal
\sum_{k_i\neq 0} \fc{1}{|k_i|^3\|k_i\a_i\|^2}&\mathop{\le}\limits_{\eqref{inequality for (k_ia_i)}} C(\a,\d)\sum_{k_i\neq 0} \fc{(\ln|k_i|)^{1+\d}}{|k_i|^2\|k_i\a_i\|}\\
&\le C(\a,\d)\sum_{k_i\neq 0} \fc{1}{|k_i|(\ln|k_i|)^{2+2\d}\|k_i\a_i\|}\\
&\mathop{\le}\limits_{\eqref{inequality for ln(k_ia_i)}} C(\a,\d)\sum_{k_i\neq 0} \fc{1}{|k_i|(\ln|k_i|)^{1+\d}\|k_i\a_i\||\ln\|k_i\a_i\||^{1+\d}}.\\
\eal
\eeq
Note that the integral 
$$
J(k_i)=\int_{\T} \fc{1}{\|k_i\a_i\| |\ln(\|k_i\a_i\|)|^{1+\d}} d\a_i
$$
is convergent and the value is the same for all $k_i$, thus for almost every $\a_i\in T$,
$$
\sum_{k_i\neq 0} \fc{1}{|k_i|(\ln |k_i|)^{1+\d}\|k_i \a_i\| |\ln \|k_i\a_i\|)|^{1+\d}}
$$
is also convergent. Then the $L^2$ convergence of $B_{\C_r}(\a,x)$ follows from the convegence of \eqref{series of convergence in L2} through \eqref{inequality for d=1}.
\end{proof}

\section{NON-RESONANT TERMS.}
This section is devoted to highlight the nodes with main contributions in the Fourier series $D_{\C,2}$, the final goal is to arrive at the sum \eqref{D'} as an equivalent expression for our Fourier series in terms of limit distributions. Throughout Section 4, we will use the formula \eqref{Fourier term for symmetric body} since we restrict ourselves to the case symmetric shapes. 

For $k=(k_1,k_2)$ and $\alpha=(\a_1,\a_2)$, we use the notation $\{k_i\alpha_i\}:=k_i\alpha_i+l_i$ where $l_i$ is the unique integer such that $-1/2<k_i\alpha_i+l_i\le1/2$. To evaluate $D_{\C,2}$, we sum up term by term in the Fourier expansion \eqref{Fourier term for symmetric body} of $\chi_{\mc{C}_r}$ for $n_1$, $n_2$, and we will simplify by using the summation formula
$$ \sum_{n=0}^{N-1} \cos(A+nB)=\frac{\cos(A+\frac{N-1}{2}B)\sin(\frac{N}{2}B)}{\sin{\frac{B_l}{2}}}.$$
Then the normalized term for the node $k$ becomes

\beq\label{f}
f(r,x,\alpha;N,k)=c_k(r)\frac{\cos(2\pi(k,x)+\pi (N-1)(\sum_{i=1}^{2}\{k_i \alpha_i\}))\prod_{i=1}^{2}\sin(\pi N\{k_i \alpha_i\})}{N^{\fc{1}{2}}\prod_{i=1}^{2}\sin(\pi\{k_i\alpha_i\})}.
\eeq
where $N^{\fc{1}{2}}$ is the normalizer.

Since the sum $D_{\C,2}$ consists of all-non-zero coordinates nodes, it becomes the following:
$$
\Delta(r,x,\alpha;N)=\sum_{k\in \mb{Z}^2:\prod_{i=1}^{2} k_i\neq 0}f(r,x,\alpha;N,k)
$$
\\
\nid $\bm{Step \ 1.}$
This step shows that the nodes outside the circle of radius $N/\e$ have a negligible combined contribution.
Given a set $S$, for funciton $h$ defined on $\l(\mb{T}^{2}\r)^2\times S$, we denote by $\|h\|_2$ the supremum of the $L^2$ norms $\|h(\cdot,s)\|$ over all $s\in S$. Let 
$$
\Delta_1(r,x,\alpha;N)=\sum_{k\in\mb{Z}^2:\forall 1\le i\le 2,\ 0<|k_i|<\frac{N}{\epsilon}}f(r,x,\alpha;N,k)
$$
\begin{lem}
We have
\beq\label{estimatebar}
\|\Delta-\Delta_1\|_2\le C\epsilon^{1/2}
\eeq
\end{lem}

\begin{proof}

Since
\[
\int_{\mb{T}}\left(\frac{\sin(\pi N(k_i \alpha_i)}{\sin (\pi (k_i \alpha_i))}\right)^2 d\alpha_i =\int_{\mb{T}}\left|\frac{e^{i\pi N k_i \alpha_i}-e^{-i\pi N k_i \alpha_i}}{e^{i\pi  k_i \alpha_i}-e^{-i\pi k_i \alpha_i}}\right|^2d\alpha_i
\]
\[
=\int_{\mb{T}}\left|\frac{e^{i2\pi N k_i \alpha_i}-1}{e^{i2\pi  k_i \alpha_i}-1}\right|^2d\alpha_i=\int_{\mb{T}}|\sum_{n=0}^{N-1}e^{i2\pi nk_i \alpha_i}|^2 d\alpha_i 
,
\]
 we have for every $1\le i \le d$, 
\[
\int_{\mb{T}}\left(\frac{\sin(\pi N(k_i \alpha_i)}{\sin (\pi (k_i \alpha_i))}\right)^2 d\alpha_i \le N
.\]
Since in the integral only the square terms have non zero contributions, and $|d_r(k)|=\mc{O}(|k|^{-\fc{3}{2}})$, we get that
\[
\|\Delta-\Delta_1\|_2^2 \le CN^2 \frac{1}{N}\sum_{|k|\ge\frac{N}{\epsilon}} \frac{1}{|k|^3}\le CN\fc{1}{\fc{N}{\e}}= C\epsilon.
\]
\end{proof}

\nid $\bm{Step \ 2.}$ We show that, within the range of $|k|<N/\epsilon$, by taking out a small measure set of $\alpha$, the divisors admit a lower bound such that $N^{\fc{1}{4}}|k_i|^{\fc{3}{4}}\{k_i\alpha_i\}> \epsilon ^{1/2}$, for every $1\le i\le 2$. Therefore we can furthur restrict our sum in the set of small divisors $S(N,\a)$ (see \eqref{Set of S(N,alpha)}).

Let

$$
E_{N}=\bigcup_{1\le|n|\le\fc{N}{\epsilon}} \left\{\alpha\in \mb{T}^2: \exists 1\le i\le 2, \quad |n|^{\fc{3}{4}}|\{n \alpha_i\}|<\fc{\epsilon^{\fc{1}{2}}}{N^{\fc{1}{4}}}\right\}.
$$

Note that
$$
|{E}_{N}|\le d\sum_{n=1}^{\fc{N}{\epsilon}} \fc{\epsilon^{\fc{1}{2}}}{|n|^{\fc{3}{4}}N^{\fc{1}{4}}}\le d\epsilon^{\fc{1}{4}}.$$

Outside the $2\epsilon^{\fc{1}{4}}$ measure set ${E}_N$, we have for $0<|k_i|< N/\epsilon$, $N^{\fc{1}{4}}|k_i|^{\fc{3}{4}}\{k_i\alpha_i\}> \epsilon ^{1/2}$, for every $1\le i\le 2$. This is how we apply the short vector argument in the next section.





Let
\beq\label{Set of S(N,alpha)}
S(N, \alpha)=\l\{ k\in\mb{Z}^2: \forall 1\le i\le 2, \quad 0<|k_i|<\frac{N}{\epsilon}, \quad |k_i|^{\fc{3}{4}}|\{k_i\alpha_i\}|<\frac{1}{\e^{2}N^{\fc{1}{4}}}.\r\},
\eeq
\beq\label{D2}
\Delta_2 (r,x,\alpha;N)=\sum_{k\in S(N,\alpha)} f(r,x,\alpha;N,k),
\eeq
We have 

\begin{lem}\label{control for big divisors}
\beq\label{inequality for big divisors}
\|\Delta-\Delta_2 \|_{L^2(\mb{T}^2\times (\mb{T}^2-E_N))}\le C\epsilon^{1/2}
\eeq
\end{lem}
\begin{proof}
By \eqref{estimatebar} it is sufficient to show that $\|\Delta_1-\Delta_2 \|_{L^2(\mb{T}^2\times (\mb{T}^{2}-E_N))}^2\le C\epsilon$. We have
$$
\|\Delta_1-\Delta_2 \|_{L^2(\mb{T}^2\times (\mb{T}^2-E_N))}^2\le \frac{C}{N }\sum_{|k|<\fc{N}{\epsilon}} A_k
$$
with 
$$
A_k=c_k^2 \int_{\mb{T}^2}\fc{1}{\prod_{i=1}^2\{k_i\alpha_i\}}\chi_{\l\{\exists 1\le i\le 2, \ |k_i|^{\fc{3}{4}}|\{k_i\alpha_i\}|\ge \frac{1}{\e^{2}N^{\fc{1}{4}}}\r\}}d\alpha.$$
We have 
$$
A_k\le c_k^2\sum_{j=1}^{2} A(k,j),
$$
where $A(k,j)$ denote the part when the $j-$coordinate violates the condition in $S(N,\a)$: 
\beq\label{decomposition for A}
\bal
A(k,j)=& \prod_{i\ne j}\sum_{p_i\ge1}\int_{\T} \fc{1}{(\{k_i\a_i\})^2} \chi_{\{p_i\e^{\fc{1}{2}}\le N^{\fc{1}{4}}|k_i|^{\fc{3}{4}}|\{k_i\alpha_i \}|\le (p_i+1)\e^{\fc{1}{2}}\}}d\a_i\\
&\times \sum_{p_j\ge 1}\int_{\T} \fc{1}{(\{k_j\a_j\})^2} \chi_{\{\fc{p_j}{\e^2}\le N^{\fc{1}{4}}|k_j|^{\fc{3}{4}}|\{k_j\alpha_j \}|\le \fc{(p_j+1)}{\e^2}\}}d\a_j\\
&=:\prod_{i\ne j}\sum_{p_i\ge1}A(k,i,p_i) \sum_{p_j\ge 1} \bar{A}(k,j,p_j)\\
\eal
\eeq
For $p_i\ge1$ we define 
$$
B(k,i,p_i)=\left\{\alpha_i\in\mb{T}:p_i\e^{\fc{1}{2}}\le N^{\fc{1}{4}}|k_i|^{\fc{3}{4}}|\{k_i\alpha_i\}|\le (p_i+1)\e^{\fc{1}{2}}\right\},
$$
and for $p_j\ge 1$, define
$$
\bar{B_l}(k,j,p_j)=\l\{\fc{p_j}{\e^2}\le N^{\fc{1}{4}}|k_j|^{\fc{3}{4}}|\{k_j\alpha_j \}|\le \fc{(p_j+1)}{\e^2}\r\}.
$$
Then 
$$|B(k,i,p_i)|\le \fc{\e^{\fc{1}{2}}}{N^{\fc{1}{4}}|k_i|^{\fc{3}{4}}}, \quad |\bar{B_l}(k,j,p_j)|\le \fc{1}{\e^2N^{\fc{1}{4}}|k_j|^{\fc{3}{4}}}.$$

Thus
\beq\label{estimation for A(k,i,p)}
A(k,i,p_i)\le \fc{\e^{\fc{1}{2}} (\Nki)^2}{(\e^{\fc{1}{2}})^2 p_i^2 \Nki}\le \fc{\Nki}{\ehalf p_i^2},
\eeq
similarly,
$$
\bar{A}(k,j,p_j)\le \fc{(\e^2)^2 (\Nkj)^2}{\e^2 p_j^2 \Nkj}\le \e^2\Nkj,
$$
By using $c_k=O\l(\fc{1}{|k|^{\fc{3}{2}}}\r)$, we obtain
$$
A_k
\le C\fc{1}{|k|^3}\e^{\fc{3}{2}}\prod_{i=1}^{2}\l(\Nki\r)\le C\e^{\fc{3}{2}}\fc{N^{\fc{1}{2}}}{|k|^{\fc{3}{2}}}
$$
Summing over k, we get
$$
\sum_{|k|<\fc{N}{\epsilon}} A_k \le C \epsilon^{\fc{3}{2}} N^{\fc{1}{2}}\sum_{|k|\le{\fc{N}{\epsilon}}} \fc{1}{|k|^{\fc{3}{2}}}
\le C \epsilon N ,
$$
and the claim follows.
\end{proof}

\nid $\bm{Step \ 3.}$ In fact, with the bounded range of $\{k_i\a_i\}$ in Step 2, we can show that the main contribution of the Fourier series comes from the nodes of coordinates of order $N$.
Let
$$ 
\hat{S}(N, \alpha)=\l\{ k\in\mb{Z}^2: \forall 1\le i\le 2, \quad N\e^{3}<|k_i|<\frac{N}{\epsilon}, \quad |k_i|^{\fc{3}{4}}|\{k_i\alpha_i\}|<\frac{1}{\e^{2}N^{\fc{1}{4}}}.\r\},
$$
\beq\label{D3}
\Delta_3 (r,x,\alpha;N)=\sum_{k\in \hat{S}(N,\alpha)} f(r,x,\alpha;N,k),
\eeq
We have 

\begin{lem}\label{control for small k}
$$
\|\Delta-\Delta_3 \|_{L^2(\mb{T}^2\times (\mb{T}^2-E_N))}\le C\epsilon^{1/2}
$$
\end{lem}
\begin{proof}
By \eqref{inequality for big divisors} it is sufficient to show that $\|\Delta_3-\Delta_2 \|_{L^2(\mb{T}^2\times (\mb{T}^2-E_N))}^2\le C\epsilon$. We have
$$
\|\Delta_3-\Delta_2 \|_{L^2(\mb{T}^2\times (\mb{T}^2-E_N))}^2\le \frac{C}{N }\sum_{|k|<N\e^{3}}\hat{A}_k
$$
with 
$$
\hat{A}_k=c_k^2 \prod_{i=1}^2\int_{\mb{T}^2}\fc{1}{\{k_i\alpha_i\}}\chi_{\l\{ |k_i|^{\fc{3}{4}}|\{k_i\alpha_i\}|\ge \frac{\e^{\fc{1}{2}}}{N^{\fc{1}{4}}}\r\}}d\alpha.$$

Repeating the argument in the Lemma \ref{control for big divisors} by replacing $\bar{A}(k,j,p_j)$ in \eqref{decomposition for A} with $A(k,i,p_i)$, and using the inequality \eqref{estimation for A(k,i,p)} we obtain
$$
\hat{A}_k
\le C\fc{1}{|k|^3}\e^{-\fc{1}{2}}\prod_{i=1}^{2}\l(\Nki\r)\le C\e^{-\fc{1}{2}}\fc{N^{\fc{1}{2}}}{|k|^{\fc{3}{2}}}
$$

Summing over $|k|\le N\e^{3}$, we get
$$
\sum_{|k|< N\e^{3}} \hat{A}_k \le C \epsilon^{-\fc{1}{2}} N^{\fc{1}{2}}\sum_{|k|< N\e^{3}} \fc{1}{|k|^{\fc{3}{2}}}
\le C \epsilon N ,
$$
and the claim follows.
\end{proof}

\nid $\bm{Step \ 4.}$
Now the error terms in the Fourier series can be savely removed. Introduce
$$
\check{f}(r,x,\alpha;N,k)=\fc{d_k(r)}{c_k(r)} f(r,x,\alpha;N,k)
$$
and let 
\beq\label{checkD}
\check{\Delta}(r,x,\alpha;N)=\sum_{k\in \check{S}(N,\alpha)} \check{f}(r,x,\alpha;N,k).
\eeq
Since $|c_k-d_k|=\mc{O}(|k|^{-\fc{5}{2}})$ and $\epsilon$ is fixed,
\beq\label{estimatecheck}
\|\check{\Delta}-\hat{\Delta}\|_{L^2(\mb{T}^2\times (\mb{T}^2-E_N))}^2\le\sum_{\epsilon^{3}N< |k|<\fc{N}{\epsilon}} \fc{C}{|k|^{d+3}} \fc{N }{N^2} \le \mc{O}(N^{-1}).
\eeq
Therefor $\hat{\Delta}$ and $\check{\Delta}$ admit the same limit distribution if there exists one.
\\

\nid $\bm{Step \ 5.}$
Observe that when $\epsilon$ is fixed, the sum in \eqref{checkD} is limited to large $k_i$ and small $\prod_{i=1}^{2}|\{k_i \alpha_i\}|$. We can replace $\check{f}$ and $\check{\Delta}$ by the following 
$$
g(r,x,\alpha;N, k)=d_k(r)\frac{\cos(2\pi(k,x)+\pi (N-1)(\sum_{i=1}^{2}\{k_i \alpha_i\}))\prod_{i=1}^{2}\sin(\pi N\{k_i \alpha_i\})}{\pi^d N^{\fc{1}{2}}\prod_{i=1}^{2}\{k_i\alpha_i\}}.
$$
Then it suffices to prove that 
\beq\label{limit for g}
\lim_{N\rightarrow\infty}\lambda\{(r,x,\alpha)\in [a,b]\times \mathbb{T}^d \times \mathbb{T}^d\ |\ \Delta'(r,x,\alpha;N)\le z\} =\mc{D}(z)
\eeq
where
\beq\label{D'}
\Delta'=\sum_{k\in U(N,\alpha)} g(r,x,\alpha;N,k)
\eeq
and $U(N,\alpha)$ is any subset of $\mb{Z}^2$ that contains $\hat{S}(N,\alpha)$.

\section{GEOMETRY OF THE SPACE OF LATTICES.}
\nid$\bm{5.1.}$
Following \cite{dani1985divergent}, Section 2, and \cite{dolgopyat2014deviations}, Section 4, we show that the set $\hat{S}(N,\alpha)$ corresponds to a set of short vectors in lattices in $M^2=M\times M$, where the lattice takes the form $L_1\times L_2$, and $L_i\in M=SL(2,\mb{R})/SL(2,\mb{Z})$. Then the discrepancy function $\Delta'$ can be seen as a function on the homogeneous space $M^2$.

Let
$$
g_T=
\begin{pmatrix}
e^{-T} & 0 \\
0 & e^{T}
\end{pmatrix}, \quad
\Lambda_{\alpha_i}=
\begin{pmatrix}
1 & 0 \\
\alpha_i &1 
\end{pmatrix}.
$$
Consider the product lattice $L(N,\alpha)=L(N,\alpha_1)\times L(N, \alpha_2)$, where $ L(N, \alpha_i)=g_{\ln N}\Lambda_{\alpha_i} \mb{Z}^2$. For each $k=(k_1,k_2)\in \mb{Z}^2$, we associate the vectors $\bm{k}_i=\bm{k}_i(k_i)=(k_i, l_i)$, where $l_i$ is the unique interger such that $-\fc{1}{2}<k_i\alpha_i + l_i \le \fc{1}{2}$. We then denote 
\beq\label{X and Z}
(X_i, Z_i)= (k_i/N, N\{k_i \alpha_i\})=g_{\ln N}\Lambda_{\alpha_i} \bm{k}_i
\eeq
We have $k \in \hat{S}(N,\alpha)$ if and only if :
\beq\label{bound of X and Z}
\epsilon^{3}< |X_i| <\fc{1}{\epsilon}, \quad |X_i|^{\fc{3}{4}}|Z_i|<\fc{1}{\e^2}
\eeq
Recall the definition of the shortest vectors $\{e_1(N,\alpha_i), \ e_2(N,\alpha_i)\}$ of $L(N, \alpha_i)$ in Section 2. We will prove a version of Lemma 4.1 in \cite{dolgopyat2014deviations} that works for our product lattice space.
\begin{lem}\label{bound for m}
For every $\epsilon >0$ there exists $K(\epsilon)>0$ such that for $\alpha$ outside $E_N$, each $k\in \hat{S}(N,\alpha)$ corresponds to a pair of unique vectors $(m^1,m^2)\in \Z^2\times\Z^2$ such that for $i=1,2$, $\|m^i\|\le K(\epsilon)$ and
$$
g_{\ln N} \Lambda_{\alpha_i} \bm{k}_i =m^{i}_{1} e_1(N, \alpha_i) + m^{i}_{2}e_2(N,\alpha_i).
$$

Conversely, for $\epsilon >0$ fixed and $N$ large enough, $\alpha\notin E_N$ implies that for each pair of vectors $(m^1,m^2)\in \Z^2\times\Z^2$, where $\|m^i\| \le K(\epsilon)$, $i=1,2$, there exists a unique $k=(k_1,k_2)\in \mb{Z}^2$ such that for $i=1, 2,$
$$
g_{\ln N} \Lambda_{\alpha_i} \bm{k}_i= (m^i, e(N, \alpha_i))=m^{i}_{1}e_1(N,\alpha_i)+m^{i}_{2}e_2(N, \alpha_i).
$$
Denote $U(N,\alpha,\epsilon)$ the set of $k=(k_1, k_2)\in \mb{Z}^2$ that corresponds to the set of pairs of vectors $\{(m^1, m^2)\in \mb{Z}^2\times\mb{Z}^2 \ |\ \|m^i\| \le K(\epsilon), \ i=1,2\}$
\end{lem}
\begin{proof}
From \eqref{bound of X and Z} we can deduce that $k\in \hat{S}(N,\alpha)$ implies $g_{\ln N} \Lambda_{\alpha_i} \bm{k}_i$ is shorter that $R(\epsilon)=\epsilon^{-\fc{17}{4}}$ for $i=1,2$. Since for each $L_i=L(N,\a_i)$, the short vectors $e_1(L_i), e_2(L_i)$ form a basis in $\mb{R}^2$, we have that the norm $\|x\|$ is equivalent to the norm $\| \sum_{j=1,2} x_j e_j(L_i)\|$. Then for every $L=L_1\times L_2$, there exists $K(L)$, such that if $m^i\in \mb{Z}^2$ satisfies $\|m^i\|\ge K(L)$, we have $\|(m^i, e(L_i))\|\ge R(\epsilon)$. Now we show that the choice of $K(L)$ can be uniform for the set of lattices $\{L\in M^2\  | \ \prod_{i=1}^{2}L(N,\alpha_i),\alpha\notin E_N \}$. Therefore it is enough to show that the set
\beq\label{set of lattice}
\{\prod_{i=1}^{2}L(N,\alpha_i),\alpha\notin E_N\}
\eeq
is precompact, since we can write the set as $\prod_{i=1}^{2} \{L(N,\alpha_i), \alpha_i\notin E_N^{(i)}\}$, we prove that each component is precompact. By the bound \eqref{bound of X and Z} for $X_i$ and $Z_i$, when $\alpha_i\notin E_N^{(i)}$, if $|X_i|<\epsilon^{3}$, then $|Z_i| \ge \fc{1}{\epsilon^{\fc{13}{4}}}$. For any $l\in \mb{Z}$, $|N(k_i\alpha_i +l)|>|N(\{k_i,\alpha_i\}|=|Z_i|$, so $|N(k_i\alpha_i+l)|$ has a lower bound, therefore all vectors in $L_i$ are longer than some $\delta$. Then by the Mahler compactness criterion for lattices\cite{raghunathan1972discrete}, the set \eqref{set of lattice} is precompact.

For the converse, when we fix $\epsilon$ and let $N$ be sufficiently large, if $\|m^i\|\le K(\epsilon)$, we have that $\|(m_i,e(N, \alpha_i))\|$ is much smaller than $N$ because of the equivalence between the two norms. For every $m^i\in \mb{Z}^2$, $(m_i,e(N,\alpha_i))=g_{\ln N} \Lambda_{\alpha_i}\bar{\bm{k}}_i$ for some unique $\bar{\bm{k}}_i=(k_i, \tilde{l}_i) \in \mb{Z}^2$, we need to show that $\tilde{l}_i=l_i$, where $l_i$ allows $-\fc{1}{2}<k_i\alpha_i+ l_i\le\fc{1}{2}$. When $\tilde{l}_i$ is not equal to $l_i$, then $|N(k_i \alpha_i+\tilde{l}_i)|\ge N/2$, contradicting the fact that $\|g_{\ln N} \Lambda_{\alpha_i}\bar{\bm{k}}_i\|$ is much smaller than $N$. Therefore $\tilde{l_i}=l_i$, and $\bar{\bm{k}}_i=\bm{k}_i$.

For each pair of vectors $\{(m^1, m^2)\in \mb{Z}^2\times\mb{Z}^2 | \ \|m^i \| \le K(\epsilon),\ i=1,2 \}$, we have the corresponding set of $\{\bm{k}_i\}_{i=1,2}$, this gives us a unique vector $k=(k_1,k_2)\in \mb{Z}^2$. Therefore the second statement follows.
\end{proof}

\nid$\bm{5.2.}$ For $m^i \in \mb{Z}^2$, $\alpha_i \in \mb{T}$, we define the coordinates of the correponding vector as
\beq\label{Def of X_m}
(m^i, e(N,\alpha_i))=(X_{m^i}, Z_{m^i})
\eeq
and define $\mathbf{m}=(m^1,m^2) \in \mb{Z}^2\times\mb{Z}^2$, $X_\mathbf{m}=(X_{m^1},X_{m^2})\in \mb{R}^2$ and $R_\mathbf{m}=\|X_\mathbf{m}\|$. We introduce a corresponding function to $g(r,x,\alpha; N,k)$ with $\mathbf{m}$ as a variable:
$$
h(r,x, \alpha;N, m)=\fc{d_r(N,n) \cos(2\pi N(X_\mathbf{m},x)+\pi \fc{N-1}{N} (\sum_{i=1}^{2}Z_{m^i})) \prod_{i=1}^{2}\sin(\pi Z_{m^i})}{R_\mathbf{m}^{\fc{3}{2}}\prod_{i=1}^{2} Z_{m^i}}
$$
with
$$
d_r(N,m)=\fc{1}{\pi^3}K^{-\fc{1}{2}}(X_{\mathbf{m}}/R_\mathbf{m}) \sin (2\pi(rNP(X_{\mathbf{m}})-\fc{1}{8})).
$$
From Section 4.1. we see that for $\alpha \notin E_N$,
$$
\sum_{\mathbf{m}\in \mb{Z}^2\times\mb{Z}^2, \|m^i\|\le K(\epsilon),i=1,2} h(r,x,\alpha;N,m)=\sum_{k\in U(N,\alpha,\epsilon)}g(r,x,\alpha;N,k),
$$
where $U(N,\alpha,\epsilon)\supset \hat{S}(N,\alpha)$.

When we restricted ourself to prime vectors $\mathbf{m}\in \Z^2\times\Z^2$, the variables $rNP(X_{\mathbf{m}})$ mod $1$ will become independent random variables that are uniformly distributed on $\T^1$. In fact, $\mathbf{m}$ could be rewritten as $\check{p}(p_1m^1,p_2m^2)$, where $\check{p}$ is the signed greatest common divisor such that the first coordinate of $(p_1m^1,p_2m^2)$ is positive, and $m^1$, $m^2$ and $p=(p_1,p_2)$ are in $\mathcal{Z}$. Since all the vectors are multiples of the prime ones, 
we introduce 
\beq\label{X_pm}
X_{p,\mathbf{m}}=(p_1X_{m^1},p_2X_{m_2}) 
\eeq
and $R_{p,\mathbf{m}}=\|X_{p,\mathbf{m}}\|$. 

Introduce
\beq\label{formula q}
\begin{aligned}
&q(r,x,\alpha;N,m,p,\check{p})=\\ 
&\fc{d_r(N,m,p,\check{p})\cos\l(2\pi\check{p}\l(\sum_{i=1}^2 \l(p_i\l(m_1,\gamma_i\l(\alpha,x,N\r)\r)\r)\r)+\pi\fc{N-1}{N}\check{p}\l(\sum_{i=1}^{2} \l(p_iZ_{m^i}\r)\r)\r)\prod_{i=1}^2\sin\l(\pi \check{p}p_iZ_{m^i}\r)}
{|\check{p}|^{\fc{7}{2}}R_{p,\mathbf{m}}^{\fc{3}{2}}\prod_{i=1}^{2}\l(p_iZ_{m^i}\r)}
\end{aligned}
\eeq
where
$$
d_r(N,m,p,\check{p})=\fc{1}{\pi^3}K^{-\fc{1}{2}}\left(\fc{X_{p,\mathbf{m}}}{R_{p,\mathbf{m}}}\right)\sin(2\pi(\check{p}rNP(X_{p,\mathbf{m}})-\fc{1}{8})),
$$
and
\beq\label{def of gamma}
\gamma_i(\alpha,x,N)=Nx_i(e_{11}(N,\alpha_i),e_{21}(N,\alpha_i)),
\eeq

where $e_{ij}$ is the $j_{th}$ coordinate of the short vector $e_i$.

Recall the definition of $\mc{Z}$ in Section 2. Remind that $\mc{Z}^2=\{\mathbf{m}=(m^1,m^2)\in \mb{Z}^{2}\times \mb{Z}^{2}\}:m^i\in \mc{Z};\ i=1,2\}$ and define $\mc{Z}_{\epsilon}=\{\mathbf{m}\in \mc{Z}^2: \|m^i\|\le K(\epsilon); \ i=1,2\}$. 
Since the set $\mc{Z}$ only consists of the primitive vectors with positive first coordinate, we need to add all the positive and negative $\check{p}$'s, the summation of $h$ above becomes a summation of $q$, note that the summation over $p$ and $m$ is finite, and the sum over $\check{p}$ is finite due to its large power:
\beq\label{sum q}
\sum_{|\check{p}|=1}^{\infty} \sum_{p\in \mathcal{Z}_{\epsilon}}\sum_{m\in \mc{Z}_{\epsilon}}q(r,x,\alpha;N,m,p,\check{p}).
\eeq

Essentially we have reformulated our discrepancy function as a function defined on the lattice space, and from Step 5 of Section 4 we have the following proposition:

\begin{prop}\label{Prop 4.2}
Assume that $(r,x,\a)$ are uniformly distributed in $X=[a,b] \times \T^2\times \T^2$, then the sum \eqref{sum q} and the normalized discrepancy function $D_{\C,2}$ in \eqref{limit in the thm 1} of Theorem 1 converge to the same law in distribution as $N\rightarrow \infty$ and then $\epsilon\rightarrow 0$, if there exists a limit law for the sum $\eqref{sum q}$.
\end{prop}

\nid$\bm{5.3.}$ \textbf{Uniform distribution of unstable submanifold $\Lambda_\alpha$}. Since for each $i$, $\Lambda_{\alpha_i}$ is the unstable submanifold under the geodesic flow $g_T$ and will become equidistributed over the whole manifold $M$, naturally the same uniform distribution law of holds in the finite product space $M^2$. We have the following proposition(see \cite{marklof2010distribution}, Theorem 5.3):
\begin{prop}\label{uniform distribution}
Denote by $\mu$ the Haar measure on $M^2$. If $\Phi: (\mb{R}^2\times \mb{R}^2)^2\times \mb{R}^2\rightarrow \mb{R}$ is a bounded continuous function, then 
\beq\label{equidistribution}
\begin{aligned}
&\lim_{N\rightarrow{\infty}} \int_{\mb{T}^2} \Phi(e(L(N,\alpha_1)), e(L(N,\alpha_2)),\alpha)d\alpha \\
& =\int_{M^2\times\mb{T}^2}\Phi(e(L_1),e(L_2),\alpha) d\mu(L_1\times  L_2)d\alpha
\end{aligned}
\eeq
\end{prop}

\section{OSCILLATING TERMS.}
In this section we will prove that typical variables appeared in the sum \eqref{sum q} will behave like independent uniformly distributed random variables.
We denote by $\mu_2$ the distribution of $e(L_1)\times   e(L_2)$ when $L=L_1\times L_2$ is distributed according to Haar measure on $M^2=\prod_\text{2 copies} SL(2,\mb{R})/SL(2,\mb{Z})$. We denote by $\lambda_{2,\epsilon}$ the Haar measure on $(\mb{T}^{2})^2\times\mb{T}^{\mathcal{Z}\times \mc{Z}}_{\epsilon}$.

The main result of this section is the following, from which the main theorem follows:
 
\begin{prop}\label{Prop5.1}
Assume that $(x,\alpha,r)$ are uniformly distributed on $\mb{T}^2\times\mb{T}^2\times[a,b]$, then the following random variables
$$
e(N,\alpha_1),\dots, e(N,\alpha_2), \quad \{\gamma_{j1}\}_{j=1}^{2}, \quad \{\gamma_{j2}\}_{j=1}^2, \quad \{A_{p,\mathbf{m}}\}_{p\in\mathcal{Z},m \in\mc{Z}_{\epsilon}}
$$
where $A_{p,\mathbf{m}}=rNP(X_{p,\mathbf{m}})$, converge in distribution as $N\rightarrow \infty$ to $\mu_2\times{\lambda}_{2,\epsilon}$
\end{prop}

In order to prove Propsition \ref{Prop5.1} in Section 6.2, we will first prove that for different vectors $\l(p^{(1)},\mathbf{m}^{(1)}\r)$,$\dots$,$\l(p^{(K)},\mathbf{m}^{(K)}\r)$ in $\mc{Z}\times\mc{Z}^2$, $\{P(X_{p^{(i)},\mathbf{m}^{(i)}})\}_{i=1}^{K}$ are typically independent over $\mb{Q}$. 
\\

\nid$\bm{6.1.}$ Exceptionally in this subsection we use the lower index for $m_i$ to represent a \emph{vector} in $\Z^2$, not to be confused with the coordinates in the \textbf{Notations} in section 2. For $\mathbf{m}=(m_{1},m_{2})\in \mb{Z}^2\times\mb{Z}^2$ with $m_{i} \in \mathbb{Z}^2$, and $p=(p_1,p_2)\in \mb{Z}^2$, $p_i\ge1$, define the function $Q_{p,\mathbf{m}}\ :\ \mb{R}^{2}\times\mb{R}^2\rightarrow \mb{R} \ :\ (z_1,z_2)\mapsto P((p_1m_1,z_1), (p_2m_2,z_2))$, where $z_i=(z_{i1},z_{i2}) \in \mb{Z}^2$ is a vector, and the bracket means euclidean inner product.

\begin{prop}\label{Q-independence of the random variables}
For different vectors $\l(p^{(1)},\mathbf{m}^{(1)}\r)$,$\dots$,$\l(p^{(K)},\mathbf{m}^{(K)}\r)$ in $\mc{Z}\times\mc{Z}^2$, if $l_1,\dots,l_K$ are such that $\sum_{i=1}^{K}l_iQ_{p^{(i)},\mathbf{m}^{(i)}}\equiv 0$, then $l_i=0$ for $i=1,\dots,K$.
\end{prop}
\begin{proof}
The proof follows the same line of reasoning of Lemma 5.2 and Proposition 5.4 in \cite{dolgopyat2014deviations}. Similarly to Lemma 5.2 in \cite{dolgopyat2014deviations}, it is easy to see that the functions $Q_{p,\mathbf{m}}$ and $P$ are real analytic and not equal to a polynomial in their variables. First we prove that the sum of terms with the same $\mathbf{m}^{(i)}$, (therefore with different $p^{(i)}$'s) must be zero.

For the first part $m_1^{(i)}$ of $\mathbf{m}^{(i)}$, let $\alpha_1=(\alpha_{11},\alpha_{12})$, $\beta_1=(\beta_{11},\beta_{12})$ $z_{1j}=\delta \alpha_{1j}+\theta\beta_{1j}$ and $z_2=(0,0)$, we have
\beq\label{factorisation of Q by malpha}
Q_{p^{(i)},\mathbf{m}^{(i)}}(z_1,z_2)=p_1^{(i)}(m_1^{(i)},\alpha_1)P\l(\delta+\fc{(m_1^{(i)},\beta_1)}{(m_1^{(i)},\alpha_1)}\theta,0\r),
\eeq
develop $\Qpmi$ with respect to $\theta$ and consider the coefficient of $\theta^2$ in the sum for $l_i\Qpmi$, we have:
\beq
\sum_{i=1}^{K} l_i\l|p_1^{(i)}\fc{(m_1^{(i)},\beta_1)^2}{(m_1^{(i)},\alpha_1)}\r|\Qpmi''(\delta)=0.
\eeq
With $m_1^{(i)}$'s all being prime vectors in $\Z^2$, we can choose a special $\a_1$ such that $|(m_1^{(i)},\alpha_1)|$ is arbitrarily small for one $i$ while all the other $|(m_1^{(j)},\alpha_1)|$ that are distinct from $(m_1^{(i)},\alpha_1)$ have a uniform lower bound, then the sum of $l_i\Qpmi$ with identical $m_1^{(i)}$ must be zero. Repeat this procedure for $m_2^{(i)}$, then the sum of $l_i\Qpmi$ with identical $\mathbf{m}^{(i)}$ must be zero. 

Next, we can assume that all $\mathbf{m}^{(i)}$ are the same. For the sake of simplicity, we assume that $m_1^{(i)}=(1,0)$ and $m_2^{(i)}=(1,0)$. First we suppose that $z_{21}=0$, choose $j$ such that $p_1^{(j)}$ is the greatest among all $p_1^{(i)}$, then

$$\Qpmi(z_1,0)=P(p_1^{(i)}z_{11}. 0)=p_1^{(j)}P\l(\fc{p_1^{(i)}}{p_1^{(j)}}z_{11}, 0\r),$$
Consider the $n$-th partial derivative of $\Qpmi$ with respect to $z_{11}$, then
$$
\fc{\partial^n}{\partial z_{11}}\Qpmi(z_1,0)= \fc{\l(p_1^{(i)}\r)^{n}}{\l(p_1^{(j)}\r)^{n-1}} \fc{\partial^n}{\partial z_{11}}P\l(\fc{p_1^{(i)}}{p_1^{(j)}}z_{11}, 0\r),
$$ 
Since $\fc{\l(p_1^{(i)}\r)^{n}}{\l(p_1^{(j)}\r)^{n-1}}<1$ for all $i\neq j$, we can take $n$ sufficiently large by analyticity of $P$, then $l_j$ becomes the dominant coefficient in the linear combination of $n$-th derivatives, we must have the linear combination of terms of identical maximal $p_1^{(j)}$ is zero. By repeating this procedure for $p_2^{(i)}$, we can deduce that the coefficient $l_i$ in front the term that has the greatest $p_2^{(i)}$ among those having the greatest $p_1^{(i)}$ is zero. Inductively, all coefficients $l_i$ are zero.
\end{proof}

By Proposition \ref{Q-independence of the random variables}, we can deduce the following: if we take a lattice $L\in M^2$ and let $z_i=(e_{11}(L_i),e_{21}(L_i))$, $i=1,2$, then $P\left(X_{p,\mathbf{m}}(L)\right)=P((p_1m_1,z_1), (p_2m_2,z_2))=Q_{p,\mathbf{m}}(z_1,z_2)$. By analyticity, for any different $\l(p^{(1)},\mathbf{m}^{(1)}\r)$,$\dots$,$\l(p^{(K)},\mathbf{m}^{(K)}\r)$ in $\mc{Z}\times\mc{Z}^2$, 
\beq\label{measure of zero set for L}
\mu\left(L:\sum_{i=1}^{K}l_i P(X_{p,\mathbf{m}}(L))=0\right)=0.
\eeq
Now by Proposition \ref{uniform distribution} we have that 
\beq\label{measure of zero set for alpha}
mes\left(\alpha\in \mb{T}^2\ :\ \left|\sum_{i=1}^K l_iP\l(X_{p^{(i)},\mathbf{m}^{(i)}}\l(L(N,\alpha)\r)\r)\right| < \epsilon \right)\rightarrow 0
\eeq
as $\epsilon\ra 0$, $N\ra\infty$. 
\\

\nid\textbf{6.2. Proof of Proposition 6.1.}
Take integers $n_{ij}, n_{21}$, $n_{12}, n_{22}$, $\{l_{p,\mathbf{m}}\}_{p\in \mc{Z}^2, m\in{\mc{Z}^2_{\epsilon}}}$ and a function $\Phi\ :\ (\mb{R}^2\times \mb{R}^2)^2 \ra \R$ of compact support. It remains to show that as $N\ra \infty$
\beq\label{equidistribution for our theorem}
\begin{aligned}
&\iiint\Phi(e(L(N,\alpha_1)),e(L(N,\alpha_2))exp\l[2\pi i \l(\sum_{j=1}^{2}(n_{j1}\gamma_{j1}+n_{j2}\gamma_{j2}) + \sum_{p\in \mc{Z}^2, m\in{\mc{Z}^2_{\epsilon}}} l_{p,\mathbf{m}} A_{p,\mathbf{m}}\r)\r]\\& 
dxd\alpha dr\ra \int_{M^2}\Phi(e(L_1),e(L_2))d\mu(L)\int_{\mb{T}^{2d}} e^{2\pi i \sum_j(n_{j1}\gamma_{j1}+n_{j2}\gamma_{j2})}d\gamma \int_{\mb{T}^{\mc{Z}^2\times \mc{Z}^2_\epsilon}} e^{2\pi i \sum_{p,\mathbf{m}} l_{p,\mathbf{m}}A_{p,\mathbf{m}}}dA,
\end{aligned}
\eeq
as $N\ra \infty$. 
\begin{proof}
This proof is very close to the proof of Proposition 5.1 in \cite{dolgopyat2014deviations}, it suffice to rewrite the original proof with the new variables and use Proposition \ref{Q-independence of the random variables} and \eqref{measure of zero set for alpha}. If for all $j$ and $p,\mathbf{m}$, $n_{j1}\equiv 0$, $n_{j2}\equiv 0$ and $l_{p,\mathbf{m}}\equiv 0$, \eqref{equidistribution for our theorem} is a special case of \eqref{equidistribution}. Then it suffice to prove \eqref{equidistribution for our theorem} in the case that at least some $n_j$ or some $l_{p,\mathbf{m}}$ are non-zero, then the right-hand side of \eqref{equidistribution for our theorem} is zero, and it reduces to the following:
\beq\label{reduced case of the prop}
\bal
\iiint\Phi(e(L(N,\alpha_1)),e(L(N,\a_2)exp\l[2\pi i \l(\sum_{j=1}^{2}(n_{j1}\gamma_{j1}+n_{j2}\gamma_{j2})+\sum_{p\in \mc{Z}^2, m\in{\mc{Z}^2_{\epsilon}}} l_{p,\mathbf{m}} A_{p,\mathbf{m}}\r)\r]&dxd\alpha dr \\
&\ra 0\\
\eal
\eeq
If $n_{j1}\neq 0$ for at least one $j$, recall that $\gamma_j(\alpha,x,N)=Nx_j(e_{11}(N,\alpha_j), e_{21}(N,\alpha_j))$, then the coefficient in front of $x_j$ in $\sum_j(n_{j1}\gamma_{j1}+n_{j2}\gamma_{j2})$ is $N(n_{j1}e_{11}(N,\alpha_j)+ n_{j2}e_{21}(N,\alpha_j))$.
Note that the coordinates $e_{11}(N,\alpha_j)$ and $e_{21}(N,\alpha_j)$ are typically $\mb{Z}$-independent outside a zero measure set of $\alpha_j$. Hence \eqref{equidistribution} implies that 
\beq\label{zero measure for alpha and n}
mes\l(\alpha_j \in\mb{T}:\l|n_{j1}e_{11}(N,\alpha_j)+ n_{j2}e_{21}(N,\alpha_j)\r|<\fc{1}{\sqrt{N}}
\r)\ra 0
\eeq
as $N\ra \infty$. This limit states that most $\a_j$ will not allow the coefficient in front of $x_j$ to be too small, then the integral of \eqref{reduced case of the prop} can be decomposed into two parts, $LHS=I_1+I_2$, where $I_1$ corresponds to the part of integral for $\alpha_j$ with $\l|n_{j1}e_{11}(N,\alpha_j)+ n_{j2}e_{21}(N,\alpha_j)\r|<\fc{1}{\sqrt{N}}$ and $I_2$ the part for $\alpha_j$ with $\l|n_{j1}e_{11}(N,\alpha_j)+ n_{j2}e_{21}(N,\alpha_j)\r|\ge\fc{1}{\sqrt{N}}$. Then 
$$
|I_1|\le Const(\Phi)mes\l(\alpha_j\in {\mb{T}}:\l|n_{j1}e_{11}(N,\alpha_j)+ n_{j2}e_{21}(N,\alpha_j)\r|<\fc{1}{\sqrt{N}}\r)
$$
so it can be arbitrarily small as $N\rightarrow \infty$ by \eqref{zero measure for alpha and n} . For $I_2$, since the coefficient of $x_j$ is not too small, we use integrate by parts with respect to $x_j$ to achieve the following estimation: 
$$
|I_2|\le\fc{Const(\Phi)}{\sqrt{N}}.
$$
Therefore this proves the case where not all $n_{j1}$ vanish, the case where not all $n_{j2}$ vanish is the same.

Similarly, if there exists some $(p,\mathbf{m})$, such that $l_{p,\mathbf{m}}$ is non-zero, we can use \eqref{measure of zero set for alpha} and integrate with respect to $r$ to obtain \eqref{reduced case of the prop}, using the same decomposition integration technique.
\end{proof}

\nid\textbf{6.3. Proof of Theorem 1(b).} Combining Proposition \ref{Prop 4.2} and Proposition \ref{Prop5.1}, by letting $N\ra \infty$ and then $\epsilon \ra 0$, we can subsitute the variables in \eqref{sum q} by uniformly distribtuted random variables on the infinite tori, thus we obtain Theorem \ref{the main thm about limit law}(b) and Proposition \ref{Prop related to thm}.
\end{spacing}

\bibliographystyle{plain}
\bibliography{references}

\end{document}